\documentclass[12pt,reqno]{amsart}
\usepackage[T1]{fontenc}
\usepackage{lmodern}
\usepackage{microtype}
\usepackage{mathtools}
\usepackage{amssymb}
\usepackage{enumitem}
\usepackage{needspace}
\usepackage[hidelinks]{hyperref}
\numberwithin{equation}{section}
\setlist[enumerate]{label=\textup{(\roman*)},leftmargin=2.2em}
\newcommand{\Sym}{\operatorname{Sym}}
\newcommand{\GL}{\operatorname{GL}}
\newcommand{\Tor}{\operatorname{Tor}}
\newcommand{\Sing}{\operatorname{Sing}}
\newtheorem{theorem}{Theorem}[section]
\newtheorem{corollary}[theorem]{Corollary}

\title[The Kemper--Malle conjecture]
{The Kemper--Malle conjecture and\\
isolated linear quotient singularities}
\author{Yang Chen}
\address{Y. Chen: School of Science, Jimei University, Xiamen, Fujian, 361021, P. R. China}
\email{chenyang1729@hotmail.com}
\date{}
\subjclass[2020]{Primary 13A50; Secondary 14B05, 14L30}
\keywords{modular invariant theory, Kemper--Malle conjecture, pseudoreflection,
direct summand property, isolated quotient singularity}
\hypersetup{pdftitle={The Kemper--Malle conjecture and isolated linear quotient singularities},pdfauthor={Yang Chen}}
\begin{document}
\raggedbottom
\begin{abstract}
We prove the Kemper--Malle conjecture, which provides a modular analogue of the Chevalley--Shephard--Todd theorem. As a consequence, the classification of isolated linear quotient singularities over an algebraically closed field of positive characteristic reduces to the classical Zassenhaus--Vincent--Wolf classification.
\end{abstract}
\maketitle

\section{Introduction}
The Chevalley--Shephard--Todd theorem originated in characteristic zero \cite{Che55,ST54}. Serre established a local version and gave a homological proof of sufficiency \cite[Theorem~$1'$ and its proof]{Serre68}. In the nonmodular case, where the group order is invertible in the ground field, the invariant ring is polynomial if and only if the group is generated by pseudoreflections \cite[Theorem~1]{Serre68}. In the modular case, polynomiality of the invariant ring implies that the group is generated by pseudoreflections, but the converse fails in general \cite[Theorem~2 and the subsequent example]{Serre68}. Kac and Watanabe gave an \'etale-topological proof of the necessary pseudoreflection condition \cite[Remark~3]{KW82}. In the same remark, Kac conjectured that, over an algebraically closed field, the invariant ring is polynomial if every point stabilizer is generated by pseudoreflections. The conjecture is false in general \cite[Introduction]{CN17}. For groups containing no transvections, Kemper and Malle proved the conjecture for irreducible representations \cite[Theorem~7.4]{KM97}, and Chen and Nan removed the irreducibility assumption \cite[Theorem~1.1]{CN17}. Serre \cite{Serre68} also proved that if a finite linear group has a polynomial invariant ring, then so does the pointwise stabilizer of every nonzero subspace (see \cite[Theorem~B(a)]{Kemper02}). Nakajima initiated a systematic study of polynomial invariants of irreducible pseudoreflection groups in positive characteristic \cite{Nak79}. Kemper and Malle proved that, for an irreducible representation, the invariant ring is polynomial if the group is generated by pseudoreflections and the pointwise stabilizer of every nonzero subspace has a polynomial invariant ring \cite[Main Theorem]{KM97}. The Kemper--Malle conjecture asserts that the same criterion holds without the irreducibility assumption (see \cite[Theorem~1.3]{Step16} and \cite[Conjecture~C]{Braun24}). The three-dimensional case was settled by Stepanov \cite{Step16} and Shchigolev--Stepanov \cite[Corollary~1.5]{SS18}. Braun proved the conjecture for arbitrary finite subgroups of $\mathrm{SL}(V)$ \cite[Theorem~A]{Braun24}. We prove this conjecture by a uniform argument for irreducible and reducible representations, without using the classification of finite reflection groups.

Let $V$ be a finite-dimensional vector space over a field $k$ and $G$ a finite linear group on $V$. Choose a basis $x_1,\ldots,x_n$ of $V^*$ and write $k[V]=\Sym_k(V^*)=k[x_1,\ldots,x_n]$. For a subspace $U\subseteq V$, write
\[
G_U=\{g\in G:gu=u,\ \forall u\in U\}.
\]
For $v\in V$, write $G_v=\{g\in G:gv=v\}$. A pseudoreflection is an element $s\in G$ for which $s-1$ has rank one.

\begin{theorem}\label{thm:main}
Suppose that $\operatorname{char}k=p>0$ and $V\ne0$. Then $k[V]^G$ is polynomial if and only if $G$ is generated by pseudoreflections and $k[V]^{G_U}$ is polynomial for every nonzero $k$-subspace $U\subseteq V$.
\end{theorem}

In the nonmodular case, the stabilizer condition is automatic for groups generated by pseudoreflections, so the criterion reduces to the classical theorem. We prove Theorem~\ref{thm:main} using the following relative form of the sufficient direction of the classical theorem. This criterion was stated by Notbohm \cite[Proposition~2.5]{Notbohm99}, who subsequently pointed out a gap in its proof \cite{Notbohm99Err}.

\begin{theorem}\label{thm:relative}
Let $k$ be a field of arbitrary characteristic and $G$ a finite linear group on $V$ generated by pseudoreflections. If a subgroup $H\leq G$ satisfies $[G:H]\in k^\times$ and $k[V]^H$ is polynomial, then $k[V]^G$ is polynomial.
\end{theorem}

No normality assumption on $H$ is needed, and taking $H$ to be trivial recovers the Chevalley--Shephard--Todd theorem. The case of normal $H$ follows from the freeness theorem of Hochster and Eagon \cite{HE71} (see \cite[Proposition~9]{BC10}). Buchaev also proved this case over algebraically closed fields of positive characteristic \cite[Proposition~2]{Buchaev26}. For irreducible representations, Theorem~\ref{thm:relative} follows from Broer's direct summand criterion \cite[Theorem~2.1]{Broer10}, whose proof uses the Kemper--Malle classification. Chen obtained a direct summand criterion for groups containing no transvections \cite[Theorem~3.3]{Chen22}. An inclusion of rings $R\subseteq A$ has the \emph{direct summand property}, or \emph{splits}, if it admits an $R$-linear retraction $A\to R$. The conditions on $H$ in Theorem~\ref{thm:relative} ensure that $k[V]^G\subseteq k[V]$ splits by freeness and relative averaging. In positive characteristic, the inclusion $k[V]^G\subseteq k[V]$ splits if and only if $k[V]^G$ is strongly $F$-regular. Here strong $F$-regularity is understood without an $F$-finiteness assumption \cite[\S9 and Theorem~7.11]{Hochster22}. Broer's conjecture, which will be treated in a separate paper, asserts that the direct summand property suffices for polynomiality of $k[V]^G$ whenever $G$ is generated by pseudoreflections~\cite{Broer05}. For $p$-groups, the following theorem complements Theorem~\ref{thm:main} and proves the conjecture of Shank and Wehlau \cite[Conjecture~1.1]{SW99} in Broer's direct summand formulation \cite[Corollary~4]{Broer05}.

\begin{theorem}\label{thm:pg}
Let $k$ be a field of characteristic $p>0$ and $P$ a finite linear $p$-group on a finite-dimensional $k$-vector space $V$. Then $k[V]^P$ is polynomial if and only if $k[V]^P\subseteq k[V]$ splits.
\end{theorem}

Together with Stepanov's theorem \cite[Theorem~1.3]{Step16}, Theorem~\ref{thm:main} reduces the classification of isolated linear quotient singularities in arbitrary dimension to the nonmodular case.

\begin{corollary}\label{thm:reduction}
Let $k$ be algebraically closed of characteristic $p>0$, $G$ a finite linear group on $V$, and $N$ the subgroup generated by the pseudoreflections in $G$. Suppose that $V/G$ is smooth away from the image of the origin. Then $k[V]^N$ is polynomial and $p\nmid[G:N]$. Let $T$ be the tangent space of $V/N$ at the image of the origin with its induced linear $(G/N)$-action. There is a $(G/N)$-equivariant $k$-isomorphism $V/N\simeq T$ sending the image of the origin to $0$. It induces a $k$-isomorphism
\[
V/G\simeq T/(G/N),
\]
where $G/N$ acts freely on $T\setminus\{0\}$.
The quotient $V/G$ is smooth if and only if $G=N$.
\end{corollary}

Together with tame uniqueness \cite[Section~2.8]{LMM25}, Corollary~\ref{thm:reduction} reduces the global, local and complete local classifications to the classical Zassenhaus--Vincent--Wolf classification of fixed-point-free representations \cite[Chs.~6 and~7]{Wolf11}, as extended to the tame positive-characteristic case in \cite[Theorem~3.13]{Step14}.

\Needspace{8\baselineskip}
\section{The Kemper--Malle criterion}\label{sec:KM}
We first prove Theorem~\ref{thm:relative} by adapting Serre's homological argument in its graded form \cite[Ch.~V, \S5, no.~2, proof of Theorem~1]{Bourbaki02}.

\begin{proof}[Proof of Theorem~\ref{thm:relative}]
Let
\[
L=\Tor_1^{k[V]^G}(k,k[V]^H),\qquad
M=\Tor_1^{k[V]^G}(k,k[V]),
\]
where $k=k[V]^G/(k[V]^G)_+$. Since $k[V]^H$ is polynomial, $k[V]$ is graded free over $k[V]^H$. Let $F_\bullet\to k$ be a graded free $k[V]^G$-resolution. Consider the natural isomorphism of complexes
\[
\begin{aligned}
\Phi:\quad k[V]\otimes_{k[V]^H}
\bigl(F_\bullet\otimes_{k[V]^G}k[V]^H\bigr)
&\xrightarrow{\sim} F_\bullet\otimes_{k[V]^G}k[V],\\
f\otimes(u\otimes b)&\longmapsto u\otimes fb.
\end{aligned}
\]
The inclusion $j:k[V]^H\hookrightarrow k[V]$ induces a map $j_*:L\to M$. Then
\[
\Phi(1\otimes z)
=(\operatorname{id}_{F_\bullet}\otimes j)(z),
\qquad \forall z\in F_\bullet\otimes_{k[V]^G}k[V]^H.
\]

Since $k[V]$ is flat over $k[V]^H$, tensoring with $k[V]$ over $k[V]^H$ commutes with homology. Thus $\Phi$ induces a $k[V]$-linear isomorphism
\[
\Phi_*:k[V]\otimes_{k[V]^H}L\xrightarrow{\sim}M,
\]
which satisfies
\[
\Phi_*(1\otimes v)=j_*(v),\qquad \forall v\in L.
\]
For $w\in M$, there exist $f_i\in k[V]$ and $v_i\in L$ such that
\[
w=\Phi_*\left(\sum_{i=1}^r f_i\otimes v_i\right)
=\sum_{i=1}^r f_i\Phi_*(1\otimes v_i)
=\sum_{i=1}^r f_i j_*(v_i),
\]
which implies that $j_*(L)$ generates $M$ as a $k[V]$-module.

The action of $G$ on $k[V]$ induces a $k[V]$-semilinear action on $M$, so $k[V]_+M$ is $G$-stable. For a pseudoreflection $s$, choose $0\ne\alpha_s\in k[V]_1$ spanning $(s-1)k[V]_1$. Since $s$ acts trivially on $k[V]/(\alpha_s)$, the map $f\mapsto(s(f)-f)/\alpha_s$ takes values in $k[V]$. It is $k[V]^G$-linear because $s$ fixes $k[V]^G$. Since $s-1$ factors through multiplication by $\alpha_s$ over $k[V]^G$, functoriality of Tor gives
\[
(s-1)M\subseteq\alpha_sM\subseteq k[V]_+M.
\]
Since $G$ is generated by pseudoreflections, $(g-1)M\subseteq k[V]_+M$ for every $g\in G$.

Let $m=[G:H]$ and define the $k[V]^G$-linear relative averaging map
\[
\rho:k[V]^H\longrightarrow k[V]^G,\qquad
\rho(b)=\frac1m\sum_{gH\in G/H}g(b).
\]
For the inclusion $i:k[V]^G\hookrightarrow k[V]$, one has $i\rho=m^{-1}\sum_{gH}g\circ j$. The induced map $(i\rho)_*:L\to M$ factors through $\Tor_1^{k[V]^G}(k,k[V]^G)=0$. Thus, for every $v\in L$,
\[
0=(i\rho)_*(v)
=\frac1m\sum_{gH\in G/H}g\bigl(j_*(v)\bigr)
\equiv j_*(v)\pmod{k[V]_+M}.
\]
Since $j_*(L)$ generates $M$ as a $k[V]$-module, $M=k[V]_+M$. The module $M$ is finitely generated and graded over $k[V]$, so the graded Nakayama lemma yields $M=0$. Hence $k[V]$ is graded free over $k[V]^G$. By faithfully flat descent of regularity \cite[Chap.~0, Proposition~17.3.3(i)]{EGAIV1}, $k[V]^G$ is regular, hence polynomial.
\end{proof}

In characteristic $p>0$, Campbell, Hughes and Pollack used relative averaging to descend Cohen--Macaulayness from the invariants of a Sylow $p$-subgroup to those of the whole group \cite[Theorem~1]{CHP91}.
We now apply Theorem~\ref{thm:relative} to pointwise stabilizers containing a Sylow $p$-subgroup.

\begin{proof}[Proof of Theorem~\ref{thm:main}]
The necessity follows from Serre's theorems. For sufficiency, let $P$ be a Sylow $p$-subgroup of $G$. Since $P$ is a $p$-group acting on the nonzero space $V$ in characteristic $p$, its fixed space $V^P$ is nonzero. Choose $0\ne v\in V^P$. Then $P\leq G_v$, so $[G:G_v]\in k^\times$. By hypothesis, $k[V]^{G_v}$ is polynomial, and Theorem~\ref{thm:relative} implies that $k[V]^G$ is polynomial.
\end{proof}

The proof of Theorem~\ref{thm:main} uses only one-dimensional subspaces. If $G$ is generated by pseudoreflections and $P$ is a Sylow $p$-subgroup, Serre's stabilizer theorem and Theorem~\ref{thm:relative} imply that $k[V]^G$ is polynomial if and only if $k[V]^{G_{V^P}}$ is polynomial. This reduces the problem to a proper subgroup when $V^G=0$. For a $p$-group $P$, Theorem~\ref{thm:main} gives no stabilizer reduction, since $P_{V^P}=P$. Earlier work established polynomiality criteria \cite{Nak83,Braun22} and settled several cases of the Shank--Wehlau--Broer conjecture \cite{BroerAb10,KumminiMondal22,KumminiMondal24,Maithani26}.

\begin{proof}[Proof of Theorem~\ref{thm:pg}]
The forward implication is standard. For the converse, we may assume that
$k$ is perfect,
since the direct summand property is preserved under field extension
and polynomiality descends to the ground field.

\emph{Claim.} If $P$ preserves
\[
 A=k[x_1^{p^{e_1}},\ldots,x_n^{p^{e_n}}]\subseteq k[V],
 \qquad 0\leq e_1\leq\cdots\leq e_n,
\]
and $A^P\subseteq A$ splits,
then $A^P$ has homogeneous polynomial generators.

We prove the claim by induction on $n=\dim V$.
Let
\[
 U=\operatorname{span}_k\{x_i:e_i=e_1\}
  =\{\ell\in V^*:\ell^{p^{e_1}}\in A\}.
\]
Since $P$ preserves $A$ and taking $p^{e_1}$-th powers
commutes with its action, $U$ is $P$-stable.
Since $P$ is a $p$-group in characteristic $p$, $U$
contains a nonzero fixed linear form.
After a change of basis in $U$, we may assume that $x_1$ is fixed. Put
\[
 N=\ker\bigl(P\longrightarrow\GL(\ker x_1)\bigr),\qquad
 B=A^N,\quad Q=P/N.
\]
Then $(g-1)V^*\subseteq kx_1$ for every $g\in N$.
Taking $N$-invariants in
$A\cap x_1k[V]=x_1^{p^{e_1}}A$ gives
\[
 B\cap x_1k[V]=x_1^{p^{e_1}}B.
\]
By \cite[Corollary~2.10]{Braun22},
restriction to $\ker x_1$ induces an isomorphism
\[
 C=B/x_1^{p^{e_1}}B\xrightarrow{\sim}k[y_2^{p^{h_2}},\ldots,y_n^{p^{h_n}}]
 \subseteq k[\ker x_1],
\]
where $y_2,\ldots,y_n$ form a basis of $(\ker x_1)^*$
and $0\leq h_2\leq\cdots\leq h_n$.
The induced $Q$-action on $C$ is faithful by injectivity of Frobenius.
The given retraction restricts to an $A^P$-linear retraction
$B\to A^P$ and induces an
$A^P/x_1^{p^{e_1}}A^P$-linear retraction
\[
 C\cong B\otimes_{A^P}\bigl(A^P/x_1^{p^{e_1}}A^P\bigr)
 \longrightarrow A^P/x_1^{p^{e_1}}A^P.
\]
Write $\operatorname{Tr}_Q(d)=\sum_{q\in Q}q(d)$ for $d\in C$.
This map is $C^Q$-linear, and its image is nonzero by independence
of distinct automorphisms.
Since $B^Q=A^P$, lifting to $B$ and taking traces gives
\[
 0\ne\operatorname{Tr}_Q(C)\subseteq A^P/x_1^{p^{e_1}}A^P.
\]
Choose $0\ne t\in\operatorname{Tr}_Q(C)$.
Since $A^P/x_1^{p^{e_1}}A^P\subseteq C$ splits,
contraction of principal ideals gives
\[
 tC^Q\subseteq tC\cap\bigl(A^P/x_1^{p^{e_1}}A^P\bigr)
 =t\bigl(A^P/x_1^{p^{e_1}}A^P\bigr).
\]
Cancelling $t$ in the domain $C$ yields
\[
 A^P/x_1^{p^{e_1}}A^P=C^Q,\qquad C^Q\subseteq C\text{ splits}.
\]
The induction hypothesis applies to $C\subseteq k[\ker x_1]$.
Lift the $n-1$ homogeneous polynomial generators of $C^Q$ to $A^P$.
Together with $x_1^{p^{e_1}}$, these generate $A^P$ by induction on degree.
They are algebraically independent because $\dim A^P=n$.
Taking all $e_i=0$ proves sufficiency in the theorem.
\end{proof}

Over any field $k$ of characteristic $p>0$, Theorem~\ref{thm:pg} therefore implies that $V/P$ is strongly $F$-regular if and only if it is smooth over $k$, in which case $V/P\simeq\mathbb A_k^{\dim V}$. Theorem~\ref{thm:relative} also detects polynomiality from the regular locus.

\Needspace{12\baselineskip}
\begin{corollary}\label{cor:singular-locus}
Let $k$ be a perfect field of characteristic $p>0$ and $G$ a finite linear group on $V\ne0$ generated by pseudoreflections. Let $P$ be a Sylow $p$-subgroup and $\pi:V\to V/G$ the quotient map. Then
\[
k[V]^G\text{ is polynomial}\quad\Longleftrightarrow\quad
\pi(V^P)\cap (V/G)_{\mathrm{reg}}\ne\varnothing.
\]
If $k[V]^G$ is not polynomial, then
\[
\dim\Sing(V/G)\geq\dim V^P\geq1.
\]
\end{corollary}
\begin{proof}
The forward implication is immediate. Conversely, applying
\cite[Theorem~C]{Kemper02} at the generic point of $V^P$ gives
polynomiality of $k[V]^{G_{V^P}}$, and hence of $k[V]^G$
by Theorem~\ref{thm:relative}.
If $k[V]^G$ is not polynomial, then
$\pi(V^P)\subseteq\Sing(V/G)$ by the equivalence.
Since $\pi$ is finite, $\dim\pi(V^P)=\dim V^P\geq1$,
giving the bound.
\end{proof}

Corollary~\ref{cor:singular-locus} rules out isolated singularities for reflection quotients, even when $k[V]^G$ is Cohen--Macaulay. It replaces $V^G$ by $V^P$ in Kemper's regularity criterion \cite[Theorem~C]{Kemper02}. For $\dim V=3$ and $p>2$, it recovers Stepanov's non-isolation theorem \cite[Theorem~1.4]{Step16}.

\Needspace{8\baselineskip}
\section{Isolated linear quotient singularities}\label{sec:reduction}
Throughout this section, $k$ is algebraically closed of characteristic $p>0$. A finite group is called \emph{tame} if its order is prime to $p$, and a representation on a nonzero vector space is \emph{fixed-point-free} if every nonidentity element fixes only the origin. For a linear quotient, an isolated singularity at the image of the origin is its only singularity, since the singular locus is closed and stable under scalar multiplication.

\begin{proof}[Proof of Corollary~\ref{thm:reduction}]
If $V/G$ is smooth, its graded coordinate ring is polynomial. Serre's necessary condition gives $G=N$, and the assertion follows. Suppose that $V/G$ is singular. By Theorem~\ref{thm:main} and Stepanov's reduction theorem \cite[Theorem 1.3 and its proof]{Step16}, the ring $k[V]^N$ is polynomial, $G/N$ is tame, and $G_v\subseteq N$ for every $v\ne0$. The residual action of $G/N$ on $V/N$ is faithful and free away from the image of the origin, since a stabilizer in $G/N$ is the image of the corresponding $G_v$.

Identify $T^*$ with $(k[V]^N)_+/(k[V]^N)_+^2$ and give it the induced grading. Since $G/N$ is tame, the natural projection $(k[V]^N)_+\to T^*$ admits a graded $k[G/N]$-linear section. The image of the section spans $(k[V]^N)_+$ modulo $(k[V]^N)_+^2$. Induction on degree therefore gives a graded $(G/N)$-equivariant surjection
\[
\Sym_k(T^*)\longrightarrow k[V]^N.
\]
Since $k[V]^N$ is polynomial, $\dim_k T=\dim(V/N)$. The surjection is an isomorphism because its source and target are domains of equal dimension. This gives the pointed $(G/N)$-equivariant $k$-isomorphism $V/N\simeq T$, and hence $V/G\simeq T/(G/N)$. Finally, if $G=N$, then $V/G$ is affine space.
\end{proof}

In dimension three, Corollary~\ref{thm:reduction} recovers the reduction statements of \cite[Corollary 1.1]{Step16} in odd characteristic and \cite[Corollary 1.6]{SS18} in characteristic two, in both cases over algebraically closed fields. Completing the quotient isomorphism also recovers the reduction statement of \cite[Theorem B]{Braun24} for $G\subseteq\mathrm{SL}(V)$.

\Needspace{6\baselineskip}
\begin{corollary}\label{cor:classification}
For $d\geq2$, the map $Q\mapsto\mathbb A_k^d/Q$ induces a bijection between
\begin{enumerate}
\item conjugacy classes of nontrivial tame subgroups $Q\leq\GL_d(k)$ acting freely on $\mathbb A_k^d\setminus\{0\}$, and
\item $k$-isomorphism classes of $d$-dimensional finite linear quotients with an isolated singularity at the image of the origin.
\end{enumerate}
The same conjugacy classes classify their local rings and completed local rings at the singular points, with all isomorphisms taken over $k$.
\end{corollary}
\begin{proof}
Corollary~\ref{thm:reduction} gives surjectivity. Conversely, let $Q$ be a nontrivial tame fixed-point-free group in dimension $d\geq2$. Freeness of its action gives smoothness of the quotient away from the origin. Since $Q$ contains no pseudoreflections, the Chevalley--Shephard--Todd theorem and the graded regularity criterion imply that the image of the origin is singular. The conjugacy class of $Q$ is recovered from the completed local ring by \cite[Section~2.8]{LMM25}. Since a global isomorphism preserves the unique singular point, this proves injectivity and the corresponding local and complete local statements.
\end{proof}

These tame conjugacy classes are given by the Zassenhaus--Vincent--Wolf classification \cite[Chs.~6 and~7]{Wolf11} in its positive-characteristic form \cite[Theorem~3.13]{Step14}. The isomorphism $V/G\simeq T/(G/N)$ in Corollary~\ref{thm:reduction} need not preserve the standard grading. The restriction $d\geq2$ is necessary, since a tame cyclic quotient of an affine line is again an affine line.

\Needspace{6\baselineskip}

\end{document}